\documentclass{article}

\usepackage{amsthm, amsmath,amsfonts,amssymb}
\usepackage{url}
\usepackage{hyperref}
\usepackage{graphicx}
\usepackage{enumerate}
\usepackage{mathtools}
\usepackage{soul} 
\usepackage{color}
\usepackage{array}

\title{Weak Infeasibility in Second Order Cone Programming \\ \ }
\newtheorem{definition}{Definition}
\newtheorem{lemma}[definition]{Lemma}
\newtheorem{proposition}[definition]{Proposition}
\newtheorem{example}[definition]{Example}

\newtheorem{theorem}[definition]{Theorem}

\newcommand{\SOC}[1]{{\mathcal{K}^{#1}}}

\newcommand{\norm}[1]{\lVert{#1}\rVert}

\newcommand{\reInt}{\mathrm{ri}\,}
\newcommand{\reCone}{\mathrm{rec}\,}
\newcommand{\lineality}{\mathrm{lin}\,}
\newcommand{\relBd}{\mathrm{relbd}\,}

\newcommand{\halfLine}{\mathrm{h}\,}

\newcommand{\inProd}[2]{\langle #1 , #2 \rangle }

\newcommand{\matRange}{{\mathrm{ range } \,}}	
\newcommand{\stdMap}{ {\mathcal{A}}}

\newcommand{\stdCone}{ {K}}

\newcommand{\pOpt}{ {\theta _P}}
\newcommand{\dOpt}{ {\theta _D}}
\newcommand{\dOptRed}{ {{\widehat \theta} _D}}
\newcommand{\dOptAlt}{ {{\theta} _{D'}}}
\newcommand{\minFace}{ {\mathcal{F}_{\min}}}

\renewcommand{\Re}{\mathbb{R}}    
\author{Bruno F. Louren\c{c}o
\thanks{
Department of Mathematical and Computing Sciences,
Tokyo Institute of Technology,
2-12-1-W8-41 Ookayama, Meguro-ku, Tokyo 152-8552, Japan. (E-mail: flourenco.b.aa@m.titech.ac.jp)}
        \and
        Masakazu Muramatsu\thanks{
                     Department of Computer Science, The University of Electro-Communications 1-5-1 Chofugaoka, Chofu-shi, Tokyo, 182-8585 Japan. (E-mail: muramatu@cs.uec.ac.jp)
                  }
         \and       
                Takashi Tsuchiya
\thanks{
National Graduate Institute for Policy Studies 7-22-1 Roppongi, Minato-ku, Tokyo 106-8677, Japan. (E-mail: tsuchiya@grips.ac.jp) \newline M. Muramatsu and T. Tsuchiya are supported in part with Grant-in-Aid for Scientific Research (B)24310112 and 
(C) 26330025. M. Muramatsu is  also partially supported by the
Grant-in-Aid for Scientific Research (B)26280005. T. Tsuchiya is also partially supported by the Grant-in-Aid for Scientific Research (B)15H02968.
                  }
        }

\begin{document}
\date{February 2015 (Revised: September 2015) }
\renewcommand{\topfraction}{0.9}
\maketitle{}
\begin{abstract}
The objective of this work is to study weak infeasibility in second order 
cone programming.
For this purpose, we consider a relaxation sequence of feasibility problems that mostly preserve
the feasibility status of the original problem. This is used to show that for a given 
weakly infeasible problem at most $m$ directions are needed to approach the cone, where 
$m$ is the number of Lorentz cones. We also tackle a closely related question and 
show that given a bounded optimization problem satisfying Slater's condition, 
we may transform it into another problem that has the same optimal value but it is ensured to attain it. From solutions 
to the new problem, we discuss how to obtain solution to the original problem which are arbitrarily close to optimality.
Finally, we discuss how to obtain finite certificate of weak infeasibility by combining our own techniques 
with facial reduction.
The analysis is similar in spirit to previous work by the authors on SDPs, but a different 
approach is required to obtain tighter bounds.  
\end{abstract}
\section{Introduction}
Second order cone programming is an important class of conic linear programming with many applications \cite{Lobo1998193}.
The problem is solved efficiently with interior-point algorithms \cite{NesterovNemirovskii,Tsuchiya99,Monteiro_Tsuchiya} as 
long as regularity conditions are satisfied. 
In this paper, we deal with the issue of weak infeasibility in second order cone 
programming (SOCP). Our starting point is the feasibility problem
\begin{equation*}
\text{find } x, \text{ such that }  x \in 	\stdCone\cap (L+c)\tag{$\mathcal{F}$}\label{socp_primal},
\end{equation*}
where $c \in \Re^n$ and $L$ is a linear subspace of $\Re^n$ and $\stdCone$ is a closed 
convex cone. When 
$\stdCone$ is a product of Lorentz cones, this is the \emph{second order cone feasibility problem} (SOCFP).
$(\stdCone ,L,c)$ will be used as shorthand for the feasibility problem \eqref{socp_primal}. 
We say that $(\stdCone ,L,c)$ is: $(i)$ \emph{strongly feasible} when 
$(L+c) \cap \reInt K \neq \emptyset$, where $\reInt$ denotes the relative interior;  $(ii)$ \emph{weakly feasible}, when $K\cap (L+c) \neq \emptyset$ but 
$(L+c) \cap \reInt K = \emptyset$; $(iii)$ \emph{weakly infeasible}, when $K\cap (L+c) = \emptyset$ but the 
distance between $\stdCone$ and $L+c$ is $0$; $(iv)$ \emph{strongly infeasible}, when $K\cap (L+c) = \emptyset$ and the distance between $L+c$ and $\stdCone$ is greater than $0$.

A major difficulty in identifying the feasibility status is the existence of weak infeasibility, since weak infeasibility
does not admit an apparent finite certificate.
A natural 
certificate of weak infeasibility is an infinite sequence $u^{(k)}$ such that 
$u^{(k)}\in L+c$ and $\lim_{k\rightarrow\infty} {\rm dist}(u^{(k)}, \stdCone) = 0$, together with 
some certificate of the infeasibility of $(\stdCone,L,c)$.  The sequence
$\{u^{(k)}\}$ is referred to as weakly infeasible sequence in this paper.

Weak infeasibility of conic programs is mainly discussed in the context of duality theory of conic programs
and regularization of ill-conditioned conic programs, e.g. \cite{Luo97dualityresults,klep_exact_2013,pataki_strong_2013,waki_how_2012}. 
In addition, it is closely related to the issue of closedness of image of closed convex cones, e.g., \cite{pataki_closedness_2007,borwein_closedness_2009}.
See \cite{lourenco_muramatsu_tsuchiya} for a more detailed review about this issue.

There is an alternative way to certificate weak infeasibility as follows.
A weakly infeasible problem is characterized as one that is infeasible but not strongly infeasible.
It is known that a system \emph{is} strongly infeasible iff the system has a dual improving direction, see Table 1 in \cite{Luo97dualityresults}.
Therefore, checking infeasibility of the original system and checking nonexistence of a dual improving direction correspond 
to two conic feasibility problems, both of which can be solved with the facial reduction algorithm. In this 
way we can detect weak infeasibility without relying on an infinite sequence. 
However, this approach is not direct in the sense that even if we know that the system is weakly infeasible,
it is not clear how to construct a weakly infeasible sequence.

In this paper, we develop a procedure of detecting weak infeasibility which enable us to construct 
a weakly infeasible sequence.
Specifically, we generate a set of at most $m$ directions and show that we are able to construct a weakly infeasible
sequence with these directions. 
A possible application of this result is to the analysis of SOCP with unattained optimal value.
Knowing the optimal value, we are able to generate an approximate optimal solution whose objective value is arbitrarily 
close to the (unattained) optimal value without solving SOCPs repeatedly.

We have three main contributions in this paper. First,  we develop a way of constructing weakly infeasible sequence and 
show that for weakly infeasible problems over a product of $m$ Lorentz cones, at most $m$ directions are needed to approach 
the second order cone. 
We will describe in Section \ref{sec:min_d} the precise meaning of that, but 
we note already that this is tighter than a recent bound obtained by Liu and Pataki \cite{liu_pataki_2015} for general 
linear conic programs. 
Another new contribution is to show how strongly feasible optimization problems 
can be further regularized in order to ensure that the optimal value is attained. There is also 
discussion on how to obtain points that are arbitrarily close to optimality for the original 
SOCP.
Finally, we discuss how to distinguish between the four different feasibility statuses, strong feasibility, weak feasibility, 
weak infeasibility and strong infeasibility without requiring any regularity condition.

The main tool we use is the set of directions $\{d^1, \ldots, d^\gamma \}$ contained in $L$ 
which are obtained through the application of  facial reduction to $(\stdCone ^*, L^\perp, 0)$ in order 
to obtain the relative interior of the feasible region of the dual system $\stdCone ^*\cap L^\perp$. 
In facial reduction theory, these directions correspond to a family of hyperplanes 
$ \{ \{d^1\}^\perp, \ldots, \{d^\gamma\}^\perp\}$ which contain  $\stdCone ^*\cap L^\perp$. As such, 
this gives an interpretation of these directions from the point of view of the dual problem.  
What is novel about our analysis is proving that these directions have other useful ``primal''
properties besides what is currently known through facial reduction theory. 
For instance, the relaxed problems induced by them have almost the same 
feasibility status as the original problem. For  weakly infeasible problems, these directions 
are useful to generate points that are arbitrarily close to the cone. This can be applied 
to strongly feasible problems with unattained optimal value, thus enabling the 
algorithmic  generation of feasible points that are close to optimality. Moreover, our sequence 
of directions is likely to be shorter than what would be otherwise obtained through plain facial reduction, 
since we show that is enough to focus on the nonlinear part of the cone.

This work is organized as follows. Section \ref{sec:notation} describes the notation 
and the setting of this work. Section \ref{sec:relaxation} discusses how to relax 
a SOCFP in a way that the feasibility properties are mostly preserved. Section \ref{sec:min_d} discussed the 
minimal number of directions needed to approach the second order cone. Section \ref{sec:unattained} contains 
a discussion on unattained strongly feasible problems and how to regularize for attainment. 
Section \ref{sec:back} describes a
theoretical recipe to distinguish the four feasibility statuses. Section \ref{sec:conc} contains a brief 
summary of this work.

\section{Notation and preliminary considerations}\label{sec:notation}
For $d \in \Re^n$, we define the closed half-space $H_d^n = \{ x \in \Re^n \mid d^Tx \geq 0 \}$
and the ray $\halfLine_d^n = \{\alpha d \in \Re^n \mid \alpha \geq 0 \}$.
We also write $x = (x_{0}, \ldots , x_{n-1})$ for 
the components of $x$. We use the notation $\overline{x}$ to denote the last 
$(n-1)$ components of $x$, i.e., $\overline{x} = (x_{1}, \ldots, x_{(n-1)})$.
The Lorentz cone in $\Re^{n}$ is denoted by 
$\SOC{n}$, i.e., $\SOC{n} = \{ x \in \Re^{n} \mid x_{0} \geq \norm{\overline{x}} \}$,
where $\norm{.}$ is the usual Euclidean norm. We remark that 
$\SOC{1} = \{x \in \Re \mid x \geq 0 \}$, so the non-negative 
orthant in $\Re^n$ can be written as a direct product of one-dimensional Lorentz cones.
If $x \in \SOC{n}$, we write $x'$ for the reflection of 
$x$ with respect to $\SOC{n}$, i.e., $x' = ( x_0, - \overline{x})$.

Our main object of study is the feasibility problem \eqref{socp_primal}, where 
$\stdCone$ is a direct product of Lorentz cones. 
We will also consider problems where $\stdCone$ also includes closed half-spaces, rays and 
subspaces. We have $\stdCone = K^{n_1} \times \ldots \times K^{n_m}$, where 
$K^{n_i} \subseteq \Re^{n_i}$ for every $i$ and $n_{1} + \ldots + n_{m} = n $.
The cone $\stdCone$ induces a block division such that for $x \in \Re^n$  we have $x \in 	\stdCone $ if 
and only if $x_{n_i} \in K^{n_{i}}$, for every $i$. 
Throughout the article we use the convention that a superscript over a set indicates 
the dimension of ambient space. For example, $\SOC{n}, H_d^n, \halfLine_d^n$ are all 
sets contained in $\Re^n$.
A single subscript under a point denotes a coordinate and double subscript 
denotes a block. For example $x_i$ is the $i$-th coordinate of $x$, while $x_{n_i}$ is 
the $i$-th block of $x$. Of course, it is implicitly understood that the division 
in blocks is induced by some cone $\stdCone \subseteq \Re^n$.  
 
Finally, we use $\reInt C$, $\reCone C$, $\lineality C$ and $\relBd C$ to denote 
the relative interior, recession cone, lineality space and relative boundary of $C$, respectively. We denote 
the dual cone of $\stdCone$ by $\stdCone^* = \{s \in \Re^n \mid s^Tx \geq 0, \forall x \in K \}$. See 
\cite{rockafellar} for basic properties of those sets.

Two convex sets $C_1,C_2 \subseteq \Re^n$ are said to be \emph{properly separated} when there is 
a hyperplane such that $C_1$ and $C_2$ lie at opposite closed half-spaces but at least 
one of them is not entirely contained in the hyperplane. A necessary and sufficient 
condition for proper separation is that $\reInt C_1 \cap \reInt C_2 = \emptyset$, see 
Theorem 11.3 of \cite{rockafellar}.
The separation is said to be \emph{strong} if there is a ball $B$ centered in the 
origin such that $C_1 +B$ and $C_2 + B$ lie at opposite open half-spaces. A necessary 
and sufficient conditions for strong separation is $\text{dist}(C_1,C_2) > 0$, see 
Theorem 11.6 of \cite{rockafellar}.

Note that strong feasibility admits an obvious certificate, since it is enough 
to obtain an element $ x \in \reInt \stdCone\cap (L+c) $. 
From these basic facts about separating hyperplanes we can characterize two of 
the remaining feasibility statuses. We have that $(K,L,c)$ is: $i)$ weakly feasible if and only if  (iff)
there is $ x \in \stdCone \cap (L+c)$ and $w \in (\stdCone ^*\setminus \stdCone^\perp)\cap L^\perp \cap \{c\}^\perp$, 
where $A^\perp$ indicates the set of elements orthogonal to a set $A$,   
$ii)$ strongly infeasible iff there is $w \in \stdCone^*\cap L^\perp$ with $w^Tc = -1$.
No such simple characterization is known for weak infeasibility, and this can be traced 
down to its asymptotic nature; in spite of the absence of 
a feasible solution, the distance between $\stdCone$ and $L+c$ is $0$. Hence, there must be some sequence $\{x^k\}$ contained in $L+c$ satisfying $\lim _{\stdCone  \to \infty} \text{dist}(x^k,\stdCone) = 0$.
However, such a sequence cannot have any converging subsequence, so $\norm{x^k} \to +\infty$.

Many of the known characterizations of weak infeasibility involve, in a way or another, 
infinite sequences (see Table 1 of Luo, Sturm and Zhang \cite{Luo96dualityand}). 
As a computer cannot verify infinite sequences, it is very hard to  distinguish numerically between weak infeasibility 
and weak feasibility,  see, for instance, P\'olik and Terlaky   \cite{polik_new_2009}. This motivates the search 
for ways of checking infeasibility without using sequences as in the recent work for semidefinite 
programming by Liu and Pataki \cite{liu_pataki_2014}, see Theorem 1 therein. See also Section 4.3 of \cite{klep_exact_2013} by Klep and Schweighofer. 
These characterizations are \emph{finite} and no infinite sequences are needed.
In Section \ref{sec:back}, we will 
show that distinguishing weak feasibility/infeasibility for SOCFPs can also be performed  without using 
sequences.

For convenience, 
we group together weak feasibility and weak infeasibility in a 
single status: \emph{weak status}. We say the feasibility status of feasibility problems $\mathcal{A}$ and $\mathcal{B}$ are ``mostly the same'' if
  $\mathcal{A}$ and $\mathcal{B}$ are  both  strongly feasible, strongly infeasible
or in weak status.  Note that it is possible that $\mathcal{A}$ is weakly infeasible and $\mathcal{B}$ is weakly feasible (or vice-versa).

\section{Relaxation of SOCFPs}\label{sec:relaxation}
In this section, we show how SOCFPs can be relaxed in a way that the feasibility 
properties are mostly preserved.
Consider a feasibility problem of the form $(\stdCone ,L,c)$, where $\stdCone$ is a direct product 
$K^{n_1} \times \ldots \times K^{n_m}$, where each $K^{n_i}$ is: the trivial cone $\{0\}$;
$\Re^{n_i}$; a second order cone $\SOC{n_i}$; a closed half-space defined by a supporting hyperplane to 
	$\SOC{n_i}$, i.e., $H_d^{n_i}$ for $d \in \SOC{n_i}\setminus \{0\}$; or
a half-line contained in $\SOC{n_i}$, i.e., $\halfLine_d^{n_i}$ for 
	$d \in \SOC{n_i}$.

Note that the family of  cones having the format above is  no more expressive 
than the family of products of second order cones. Still, for our purposes we need 
to consider this slightly more general situation because these cones will 
appear as  byproducts of Theorem \ref{theorem_reduction}. We will call them \emph{extended 
second order cones}. We remark that the 
dual cone $K^*$ is the direct product of the duals of the cones $K^{n_i}$ and it is also 
an extended second order cone. It is 
also clear that we have $(H_d^{n_i})^*   = \halfLine_d^{n_i}$.

Suppose that we have a non-zero element $a \in K$, then 
we  define: $i)\, \mathcal{H}_{1}(a,K) = \{i \mid K^{n_{i}} = \SOC{n_{i}},  a _{n_{i}} \in \reInt \SOC{n_{i}} \}$ and 
$ii)\, \mathcal{H}_{2}(a,K) = \{i \mid K^{n_{i}} = \SOC{n_{i}},  a _{n_{i}} \in (\relBd \SOC{n_{i}})\setminus \{0\}\}$.
We will omit $\stdCone$ when it is clear from the context.

\begin{lemma}\label{lemma_recession}
Let $x \in \Re^n$ and $a \in \SOC{n}$ be such that $x^Ta' > 0$. Then
$x + ta \in \reInt \SOC{n}$ for $t > 0$ sufficiently large. 
\end{lemma}
\begin{proof}
The point $a$ must be non-zero and if it is an interior point, then 
the statement clearly holds. If $a$ lies in the boundary, then
\begin{equation*}
(x + ta)_{0}^{2} - \norm{\overline{x + ta}}^{2}  = 2t(a_{0}x_{0}  - \overline{a}^{T}\overline{x}) + x_{0}^{2} - \norm{\overline{x}}^{2}.
\end{equation*}
However, $a_{0}x_{0}  - \overline{a}^{T}\overline{x}$ is equal to 
$ x^Ta'$. So if $t$ is large enough we have that $(x + ta)_{0}^{2} - \norm{\overline{x + ta}}^{2}$ 
will be greater than $0$.
\end{proof}

\begin{theorem}\label{theorem_reduction}
Let $(\stdCone ,L,c)$ be a feasibility problem such that 
$K = K^{n_1} \times \ldots \times K^{n_m}$. Suppose that 
there is  $a \in K \cap L$ such that $\mathcal{H}_{1}(a) \cup \mathcal{H}_{2}(a)$ is 
non-empty. Define the cone $\widetilde{\stdCone } = \widetilde{\stdCone }^{n_1} \times \ldots \times \widetilde{\stdCone }^{n_m}$ such 
that for every $i$: 
\begin{itemize}
	\item $\widetilde{\stdCone }^{n_i} = \Re^{n_i}$ if $i \in \mathcal{H}_{1}(a)$,
	\item $\widetilde{\stdCone }^{n_i} = H_{d}^{n_i}$ where $d = a_{n_i}'$, if $i \in \mathcal{H}_{2}(a)$,
	\item $\widetilde{\stdCone }^{n_i} = K^{n_i}$, otherwise. 
\end{itemize}
Then 
\begin{enumerate}[i.]
	\item  $(\stdCone ,L,c)$ is strongly feasible if and only if $(\widetilde{\stdCone },L,c)$ is strongly feasible;
	\item  $(\stdCone ,L,c)$ is in weak status if and only if $(\widetilde{\stdCone },L,c)$ is in weak status;
	\item  $(\stdCone ,L,c)$ is strongly infeasible if and only if $(\widetilde{\stdCone },L,c)$ is strongly infeasible.
\end{enumerate}
\end{theorem}
\begin{proof}
$(i)$ If $(\stdCone ,L,c)$ is strongly feasible, then for a relative interior point 
$y \in L+c$, we have $y_{n_{i}}^{T}a'_{n_{i}} > 0$, for all $i \in \mathcal{H}_{2}(a)$. All 
the other coordinate blocks of $y_{n_i}$ stay in the relative interior of the respective 
cones. So, $(\widetilde{\stdCone },L,c)$ is strongly feasible.

Now, if $(\widetilde{\stdCone },L,c)$  is strongly feasible
we pick $y \in L+c$  such that $y$ lies in the relative interior of $\widetilde{\stdCone }$.
For $i \in \mathcal{H}_{1}(a)$  we have $a_{n_{i}} \in \reInt \SOC{n_i}$
and for $i \in \mathcal{H}_{2}(a)$ we have $y_{n_{i}}^{T}a'_{n_{i}} > 0$. 
Hence if $t$ is sufficiently large we have $(y+ta)_{n_i}^{T}a'_{n_{i}} \in \text{int}(\SOC{n_i})$, 
for all $i \in \mathcal{H}_{1}(a) \cup\mathcal{H}_{2}(a) $, by Lemma \ref{lemma_recession}. It is also clear that 
adding $ta$ does not affect the fact that $y_{n_i} \in \text{ri} (\stdCone ^{n_i})$ for $i \not \in \mathcal{H}_{1}(a)\cup \mathcal{H}_{2}(a)$.

$(iii)$ If $(\widetilde{\stdCone }, L,c)$ is strongly infeasible then $(\stdCone ,L,c)$ also 
is because $K \subseteq \widetilde{\stdCone }$. Let us prove the converse now.
We have that $(\stdCone ,L,c)$ is strongly infeasible if 
and only if  there exists $s$ such that $s \in L^{\perp}\cap K^{*}$ and 
$s^{T}c < 0$ (see Lemma 5 of \cite{Luo97dualityresults}). In particular, $s^Ta = 0$. This means that $s_{n_i}^Ta_{n_i} = 0$ for every 
$i$, because $s \in K^*$ and $a \in K\cap L$.
It follows that for $i \in \mathcal{H}_{1}(a)$ we have $s_{n_{i}} = 0$. Also, 
for $i \in \mathcal{H}_{2}(a)$ we have that $s_{n_{i}}$ is a non-negative multiple of  $a_{n_{i}}'$ (including, of 
course, the possibility that $s_{n_{i}}$ is $0$)\footnote{Recall that if $x,y \in \SOC{n}$ satisfy $x^Ty = 0$, then $x_0\overline{y} + y_0\overline{x} = 0$. }.
We conclude that $s$ also produces strong separation for $(\widetilde{\stdCone },L,c)$ because 
$s \in \widetilde{\stdCone }^*$. So $(\widetilde{\stdCone }, L,c)$ is strongly infeasible.

Finally, $(ii)$ follows by elimination.

\end{proof}

\subsection{Relaxation sequence}
After applying Theorem \ref{theorem_reduction} to $(\stdCone ,L,c)$, it might  still 
be possible to relax it further. This motivates the next definition.

\begin{definition}[Relaxation sequence]
A relaxation sequence for $(\stdCone ,L,c)$ is a finite sequence 
of conic feasibility problems $\{ (\stdCone _{1},L,c), \ldots ,  (\stdCone _{\gamma},L,c)\}$ such that $K_1 = K$ and:
\begin{enumerate}
	\item Every $K_i$ is an extended second order cone, see the beginning of Section \ref{sec:relaxation}.
	\item For $i > 1$, there is $d^{i-1} \in K_{i-1}\cap L$ such that 
	$(\stdCone _{i},L,c)$ is obtained as a result of applying Theorem \ref{theorem_reduction} to 
	$K_{i-1}\cap L $ and $d^{i-1}$. In addition, we must have 
	$K_{i-1} \subsetneq K_{i}$ (i.e., we do not admit trivial relaxations).
\end{enumerate}
The vectors in $\{d^{1}, \ldots , d^{\gamma -1} \}$ are called \emph{reducing directions}, due 
to the fact that they came from the application of facial reduction to the dual system $(\stdCone^*,L^\perp,0)$.
A relaxation sequence  is \emph{maximal} if it does not admit non-trivial relaxations.
The problem $(\stdCone _{\gamma},L,c)$ is called \emph{the last problem} of the sequence.
The length of the sequence is defined to be $\gamma$.
\end{definition}
We now attempt to clarify the connection between 
relaxation sequences and facial reduction. Given a conic linear program 
$(\stdCone ,L,c)$, facial reduction algorithms (FRAs) \cite{Borwein1981495,borwein_facial_1981,pataki_strong_2013,article_waki_muramatsu} aim at identifying the 
minimal face $\minFace$ of $\stdCone$ which contains the feasible region $K\cap (L+c)$. 
This is done by generating a sequence of faces ending at $\minFace$.
In this respect, 
FRA and relaxation sequences accomplish different goals. However, it can 
be shown that the cones appearing in a relaxation sequence correspond to the 
\emph{dual} of the \emph{faces} obtained by applying FRA to $(\stdCone ^*,L^\perp, 0)$. 
As we are considering dual of faces instead of the faces themselves, this is more akin to 
the conic expansion algorithm as described in Section 4 of \cite{article_waki_muramatsu}.
While this allows us to cast our techniques in the conic expansion framework, as it is 
a more indirect route, it seems that
not much geometric insight is gained by doing that. Moreover, it is not obvious that results 
such as Proposition \ref{prop_last_problem} and Theorem \ref{theorem_directions} hold.  In the next two sections, we will 
use relaxation sequences to prove basic properties of weakly infeasible problems and 
unattained problems.

Since every reducing direction is responsible for relaxing at 
least one Lorentz cone, the maximum length of a relaxation sequence is 
$m+1$, where $m$ is the number of second order cones appearing in $\stdCone$.
Each relaxed problem almost preserves the feasibility status of 
the original, in the sense of Theorem \ref{theorem_reduction}. We will prove that when the relaxation sequence is maximal, 
the last problem cannot be weakly infeasible.
Before we go further, we need  a detour which we believe might 
be of independent interest.
\begin{theorem}\label{theo_lineality}
Let $C_{1}$ and $C_{2}$ be non-empty convex sets in $\Re^{n}$ such that 
$C_{1}$ is polyhedral, $C_{2}$ is closed. Suppose 
that 
\begin{equation*}
\reCone C_1\cap - \reCone C_2 \subseteq \lineality C_2,
\end{equation*}
where  $\reCone C = \{x \in \Re^n \mid x + C \subseteq C \}$ is the  recession cone of a closed convex set $C$.
Then $C_1 + C_2$ is closed.
\end{theorem}
\begin{proof}
See Theorem 20.3 in \cite{rockafellar}.	
\end{proof}
We will show that if $C_2$ is the direct product of a closed convex set 
and a polyhedral set,  we may weaken the assumptions of the Theorem \ref{theo_lineality}.

\begin{proposition}\label{prop_lineality}
Let $C_{1}$ and $C_{2}\times P$ be non-empty convex sets in $\Re^{n}$ such that 
$C_{1}$ and $P$ are polyhedral,  and $C_{2}$ is closed. Suppose that
\begin{equation}\label{eq_prop_lineality}
\reCone C_1\cap - ( \reCone C_2\times \reCone P) \subseteq \lineality C_2 \times -\reCone P.
\end{equation}
Then $C_1 + (C_2\times P)$ is closed.
\end{proposition}

\begin{proof}
We have that $C_1 + C_2\times P = (C_1 + \{0\}\times P) + C_2\times \{0\}$.
Since $C_1$ and $P$ are polyhedral sets, $(C_1 + (\{0\}\times P))$ is also polyhedral. 
We would 
like to use Theorem \ref{theo_lineality} with $(C_1 + \{0\}\times P)$  and $C_2\times \{0\} $. For 
that purpose, we are required to check that 
\begin{equation}\label{eq_prop_lineality2}
(\reCone C_1 + (\{0\}\times \reCone P))\cap -(\reCone C_2\times \{0\}) \subseteq \lineality C_2 \times \{0\},
\end{equation}
because, due to polyhedrality, $\reCone (C_1 + (\{0\}\times P))  = \reCone C_1 + (\{0\}\times \reCone P)$.
Let $(x,y)$ be a point that belongs to the set at the left-hand side of Equation \eqref{eq_prop_lineality2}, 
then $x \in - \reCone C_2 $ and $y = a+p = 0$, where $p \in \reCone P$ and 
$(x,a) \in  \reCone C_1$. It follows that $(x,a) \in -(\reCone C_2 \times \reCone P)$. Since 
we are under the assumption that Equation \eqref{eq_prop_lineality} holds, $x \in \lineality C_2$. 
Hence, $(x,y) \in  \lineality C_2 \times \{0\}$ and we are done.
\end{proof}
The following proposition is a small modification of Corollary 20.3.1 of \cite{rockafellar}.
\begin{proposition}\label{prop_strong_separation}
Let $C_{1}$ and $C_{2}\times P$ be non-empty convex sets in $\Re^{n}$ such that 
$C_{1}$ and $P$ are polyhedral, and $C_{2}$ is closed. Suppose that
\begin{equation}\label{eq_prop_lineality_separation}
\reCone C_1\cap (\reCone C_2\times \reCone P) \subseteq \lineality C_2 \times \reCone P.
\end{equation}
and that $C_1 \cap (C_{2}\times P) = \emptyset$. Then $C_1$ and $C_{2}\times P$ 
can be strongly separated.
\end{proposition}
\begin{proof}
Since $C_1 \cap( C_{2}\times P) = \emptyset$, we have that $0 \not \in C_1 - (C_{2}\times P)$.
Applying Proposition \ref{prop_lineality} to $C_1$ and $-(C_{2}\times P)$ we find 
that $C_1 - (C_{2}\times P) $ is closed. Therefore, both sets can be strongly separated.
\end{proof}

\begin{proposition}\label{prop_last_problem}
If 	$\{ (\stdCone _{1},L,c), \ldots ,  (\stdCone _{\gamma},L,c)\}$ is a maximal relaxation 
sequence for $(\stdCone ,L,c)$ then we have:
\begin{enumerate}[i.]
	\item  $(\stdCone ,L,c)$ is strongly feasible if and only if $(\stdCone _{\gamma},L,c)$ is strongly feasible;
	\item  $(\stdCone ,L,c)$  is in weak status if and only if  $(\stdCone _{\gamma},L,c)$ is weakly feasible;
	\item  $(\stdCone ,L,c)$ is strongly infeasible if and only if $(\stdCone _{\gamma},L,c)$ is strongly infeasible.
\end{enumerate}
\end{proposition}
\begin{proof}
By induction and using Theorem \ref{theorem_reduction}, items $(i)$ and 
$(iii)$ follow. We can also conclude that $(\stdCone ,L,c)$  
is in weak status if and only if  $(\stdCone _{\gamma},L,c)$ is in weak status. 
Now, suppose that $(\stdCone _{\gamma},L,c)$  is infeasible and that $(\stdCone ,L,c)$ is in weak status.
 To finish the proof, we have to show that $(\stdCone _{\gamma},L,c)$ cannot be weakly infeasible.

Reordering if necessary, 
we may assume that $K_\gamma = \widetilde{\stdCone } \times \widetilde P$, where $\widetilde{\stdCone } $  is the direct 
product of Lorentz cones and $P$ is a polyhedral cone. In this case, $\widetilde P$ is a 
direct product of half-spaces and vector spaces. Now, 
we would like to use Proposition \ref{prop_strong_separation} by setting $C_{1} = L+c$, $C_{2} = \widetilde{\stdCone }$ 
and $P = \widetilde P$.
Let us check that Equation \eqref{eq_prop_lineality_separation} is satisfied. We have 
$\reCone C_1\cap (\reCone C_2\times \reCone P)  = L\cap (\widetilde{\stdCone } \times \widetilde P)$
and $\lineality C_2 \times \reCone P  = \{0\} \times \widetilde P$.

Pick an element $x \in L\cap (\widetilde{\stdCone } \times \widetilde P)$. We must 
have $x \in \{0\}\times P$, otherwise we would be able to apply Proposition 
\ref{theorem_reduction} one more time, which would contradict the assumption of maximality.
Since Equation \eqref{eq_prop_lineality_separation} is satisfied, it follows 
that if $(\stdCone _{\gamma},L,c)$ is infeasible, it must be strongly infeasible. 

\end{proof}

\begin{example}\label{example_relaxation}
Let $(\stdCone ,L,c)$ be such that $K  = \SOC{3} \times \SOC{3}$ and 
$L+c =  \{ (t,t,s)\times (s,s,1) \mid (t,t,s) \in \SOC{3} ,  (s,s,1) \in \SOC{3} \}$.
Then, $a = (1,1,0) \times (0,0,0)  \in K\cap L$. Thus, we can relax the cone 
constraint from $\SOC{3} \times \SOC{3}$ to  $H^{3}_{a_{n_1}'} \times \SOC{3}$. Now, 
$b = (0,0,1) \times (1,1,0) \in H^{3}_{a_{n_1}'} \times \SOC{3}$. Thus, we can relax the problem 
from $H^{3}_{a_{n_1}'} \times \SOC{3}$ to $H^{3}_{a_{n_1}'} \times H^{3}_{b_{n_2}'}$.
The problem $(H^{3}_{a_{n_1}'} \times H^{3}_{b_{n_2}'},L,c)$ is weakly feasible, because 
no point in $L+c$  strictly satisfies the inequalities which define $H^{3}_{a_{n_1}'} \times H^{3}_{b_{n_2}'}$.
This implies that $(\stdCone ,L,c)$ is in weak status.
\end{example}

\section{The minimal number of directions needed to approach $\stdCone$}\label{sec:min_d}
\newcommand{\SymMatrices}[1]{{\mathcal{S}^{#1}}}
\newcommand{\PSDcone}[1]{{\mathcal{S}^{#1}_+}}
Let $\stdCone$ be an extended second order cone, therefore it is a direct product 
of Lorentz cones and polyhedral cones. Suppose that there are $m$ Lorentz cones among them.
In this section, we will show that given a weakly infeasible feasibility problem $(\stdCone ,L,c)$ there is 
$c' \in \Re^n$, a subspace $L'$ contained in $L$ of dimension at most $m$ such 
that $(\stdCone ,L',c')$ is weakly infeasible.  
This means that starting at $c'$, at most $m$ directions are 
needed to approach the cone. One application of this result is on the study 
of problems with unattained optimal value, as in Section \ref{sec:unattained}.
Note that, \textit{a priori},  the number of direction needed to approach the cone could be up to the dimension 
of the affine space $L+c$. Theorem \ref{theorem_directions} states, however, it is bounded by $m$, regardless of
the dimension of $L+c$.
\begin{theorem}\label{theorem_directions}
Let $(\stdCone ,L,c)$ be a weakly infeasible problem. 
Then there are a subspace $L'\subseteq L$ 
and $c' \in L+c$ such that $(\stdCone ,L',c')$ is weakly infeasible and 
dimension of $L'+c'$ is at most $m$, where $m$ is the number of Lorentz cones.
\end{theorem}
\begin{proof}
Let $\{ (\stdCone _{1},L,c), \ldots ,  (\stdCone _{\gamma},L,c)\}$ be a maximal 
relaxation sequence and  $\{d^{1}, \ldots , d^{\gamma -1} \}$ the 
associated set of reducing directions.  Each $d^i$ 
is responsible for relaxing at least one Lorentz cone. Since there are most $m$ of them,
there are at most $m$ directions. Due to Proposition \ref{prop_last_problem}, 
the last problem is weakly feasible, so it admits a feasible solution 
$c'$.

If $L'$ is the space spanned by $\{d^{1}, \ldots , d^{\gamma -1} \}$ then 
$(\stdCone ,L',c')$ is weakly infeasible. After all, $(\stdCone ,L',c')$ shares 
the same maximal relaxation sequence and Proposition \ref{prop_last_problem} implies 
that $(\stdCone ,L',c')$ has weak status. Also,   $L'+c'$ is an affine subspace 
of $L+c$, so $(\stdCone ,L',c')$ is an infeasible problem. 
\end{proof}
Let $\PSDcone{n}$ denote the cone of $n\times n$ positive semidefinite matrices.
In \cite{lourenco_muramatsu_tsuchiya}, it was shown that given a weakly infeasible 
semidefinite feasibility problem  (SDFP) $(\PSDcone{n},L,c)$, there is an affine space $L'+c'$ contained in $L+c$
of dimension at most $n-1$ such that $(\PSDcone{n},L',c')$ is weakly infeasible. Transforming a weakly infeasible 
SOCFP into a $n\times n$ dimensional SDFP, immediately yields the bound $n-1$.  
Note that this a much worse bound since $n-1$ is typically  larger than the number 
of Lorentz cones $m$.

Recently, Liu and Pataki generalized the results in \cite{lourenco_muramatsu_tsuchiya} and showed 
that the dimension of $L'+c'$ can be taken to be less or equal than $\ell _{K^*} -1$, see items $ii.$ and $iii.$ of 
Theorem 9 in \cite{liu_pataki_2015}. The quantity $\ell _{K^*}$ corresponds to the length of 
the longest chain of faces of $\stdCone^*$. A chain of faces of $\stdCone$ is a finite sequence of 
faces satisfying $F_1\subsetneq \ldots \subsetneq F_\ell$ and the length is defined to 
be $\ell$. For the SDP case, they showed that the bound can be refined to 
$\ell _\PSDcone{n} -2$, which matches the bound discussed in \cite{lourenco_muramatsu_tsuchiya}, since 
$\ell _\PSDcone{n} = n+1$.

If $K = \SOC{n_1}\times \ldots \times \SOC{n_m}$, then the largest chain of 
faces of $K^*$ has length $2m + 1$. Liu and Pataki's result implies the bound 
$2m$ on the dimension of  $L'+c'$, which is too pessimistic in view of 
Theorem \ref{theorem_reduction}.

\section{Finding the optimal value in a strongly feasible unattained problem}\label{sec:unattained}
Consider a pair of primal and dual SOCPs problems:
\begin{align}
\underset{x}{\inf} & \quad c^Tx \label{eq:primal}\tag{P}\\ 
\mbox{subject to} & \quad \stdMap x = b, \quad x \in K^* \nonumber \\ \nonumber\\
\underset{y}{\sup} & \quad b^Ty \label{eq:dual} \tag{D} \\ 
\mbox{subject to} & \quad c - \stdMap ^Ty \in K \nonumber,
\end{align}
where $\stdMap : \Re^n \to \Re^m$ is linear map, $b \in \Re^m $, $c \in \Re^n$ and  
$\stdMap ^T$ is the adjoint of $\stdMap$.  We will denote by $\pOpt$ and $\dOpt$, the primal and dual optimal values, respectively.
Now, suppose that one of them (but not both) satisfy Slater's condition. This assumption can be satisfied by 
applying facial reduction to \eqref{eq:dual}, for instance.
Under these conditions, if the objective function of \eqref{eq:dual} is bounded above, 
then $\pOpt = \dOpt$ and the primal optimal value is attained. However, the dual optimal 
value may not be attained. It is natural then to consider whether  \eqref{eq:primal} and \eqref{eq:dual}
can be regularized so that the common optimal value is attained at both sides. Moreover, 
one may be interested in finding dual feasible points for \eqref{eq:dual} which are arbitrarily close to optimality. 
In this section, we will show how to use the techniques developed so far to accomplish both tasks.

The first step is to consider the subspace $L = (\matRange \stdMap ^T) \cap \{ \stdMap^T y \mid \inProd{b}{y} = 0 \}$.
Let $a = -\stdMap ^T y^1   \in L\cap \stdCone$ be a nonzero point and let $\widetilde \stdCone$ be cone obtained 
as a result of applying Theorem \ref{theorem_reduction} to $(\stdCone ,L,c)$. 
We have the following lemma.
\begin{lemma}\label{lemma:preserv}
Suppose that \eqref{eq:dual} is strongly feasible, i.e., there is $y$ such 
that $c - \stdMap^Ty \in \reInt \stdCone$.
Let $\dOptRed = \sup \{b^Ty \mid c- \stdMap ^T y \in \widetilde \stdCone \}$. 
Then $\dOpt = \dOptRed$. 
\end{lemma}
\begin{proof}
As $\widetilde \stdCone \supseteq K$, we have $\dOpt \leq \dOptRed$. We will now show 
that $\dOpt \geq \dOptRed$ holds as well. The first observation is that
Theorem \ref{theorem_reduction} ensures that $(\widetilde \stdCone, \matRange \stdMap^T, c) $ is 
strongly feasible as well.  This, together with the definition of $\dOptRed$,
implies that for every  $\mu <  \dOptRed$, there is 
$y_\mu$ such that $s_\mu = c-\stdMap^Ty_\mu \in \reInt \widetilde K$ and 
$\inProd{b}{y_\mu} \geq \mu$.

Because $s_\mu$ is a relative interior point of $\widetilde K$, if $\alpha _1$ 
is positive and sufficiently large, it will be the case that $s_\mu + \alpha _1 a^1 \in \reInt \stdCone$, by 
Lemma \ref{lemma_recession}. As $s_\mu + \alpha _1 a^1 =  c-\stdMap^T(y_\mu + \alpha _1 y^1)$ and 
$\inProd{b}{y^1} = 0$, we conclude that $y_\mu + \alpha _1 y^1$ is a feasible solution 
for \eqref{eq:dual} whose value is at least $\mu$. This readily shows that 
$\dOptRed = \dOpt$.

\end{proof}

Under the conditions of Lemma \ref{lemma:preserv}, it is possible that the optimal 
value of the relaxed problem is attained even if $\dOpt$ is not attained for \eqref{eq:dual}. 
We will  now show that if former is attained, then it is possible to construct solutions close to optimality
for the latter in a natural way. The recipe is as follows. Suppose that $y^*$ is an 
optimal solution for $\sup \{\inProd{b}{y} \mid c- \stdMap ^T y \in \widetilde \stdCone \}$ and let 
$\hat y$ be \emph{any} point such that $c- \stdMap ^T \hat y \in \reInt \widetilde \stdCone$. For every 
$\beta \in [0,1)$, the point $s_\beta = c- \stdMap ^T ((1-\beta) \hat y + \beta y^*)$ lies in 
the relative interior of $\widetilde \stdCone$. Following the proof of Lemma \ref{lemma:preserv}, for 
fixed $\beta$, there will be some $\alpha^1$ such that $s_\beta + \alpha _1 a^1 \in \stdCone$.
Note that $s_\beta + \alpha _1 a^1$ corresponds to a feasible solution having value 
$\inProd{b}{(1-\beta) \hat y + \beta y^*}$. As $\beta$ goes to $1$, $s_\beta + \alpha _1 a^1$ approaches 
optimality, at the cost of, perhaps, making $\alpha ^1$ large. The next step is to 
show that if we keep relaxing the problem, we will must eventually reach some problem 
whose optimal value is attained if $\dOpt$ is finite.

\begin{theorem}\label{theo:attainment}
Suppose that \eqref{eq:dual} is strongly feasible. Let $L = (\matRange \stdMap ^T) \cap \{ \stdMap^T y \mid b^Ty = 0 \}$ and 
consider a maximal relaxation sequence for $(\stdCone ,L,c)$. 
Let $\stdCone _\gamma$ be the cone corresponding to the last subproblem. Consider the following SOCP in dual format.
\begin{align}
\underset{y}{\sup} & \quad b^Ty \label{eq:dual_red} \tag{D'} \\ 
\mbox{subject to} & \quad c - \stdMap ^Ty \in K_\gamma \nonumber,
\end{align}
The following properties hold.
\begin{enumerate}[$i.$]
\item \eqref{eq:dual_red} has  a relative interior feasible solution, so the corresponding primal problem 
(P') is attained.
\item The optimal value $\dOptAlt$ of \eqref{eq:dual_red}  satisfies $\dOptAlt = \dOpt$.  If 
$\dOptAlt$ is finite, then it is attained. 
\end{enumerate}
\end{theorem}
\begin{proof}
Successive applications of Theorem \ref{theorem_reduction} yield the first item. Using Lemma 
\ref{lemma:preserv} and induction, we can also conclude that $\dOptAlt = \dOpt$. We now 
suppose that $\dOptAlt$ is finite and assume, for the sake of contradiction, that the optimal 
value is not attained for \eqref{eq:dual_red}.

We first observe that the affine space $\mathcal{F} = (c + \matRange{\stdMap^T})\cap(\{ \stdMap^T y \mid b^Ty = \dOptAlt \} $
is non-empty. Also, the Euclidean distance between $\mathcal{F}$ and $K_\gamma$ must be zero. Both observations 
follow from the definition of $\dOptAlt$, which ensures the existence of a sequence $\{y^k \}$ contained in $(c + \matRange{\stdMap^T})\cap K_\gamma$
which satisfies $b^T{y^k} \to \dOptAlt$.

Having observed that, we will proceed as in the proof of item $ii$ of Proposition \ref{prop_last_problem}.
Since there is no optimal solution for \eqref{eq:dual_red}, we have $\mathcal{F} \cap \stdCone_\gamma = \emptyset$. Reordering 
the coordinates if necessary, we can write $\stdCone_\gamma$ as the direct product of a closed convex 
cone $\widetilde \stdCone$ and a polyhedral cone $\widetilde P$. Note that the recession cone of $\mathcal{F}$ is 
$L$ and the recession cone of $\stdCone_\gamma$ is $\stdCone_\gamma$. By the maximality of the relaxation sequence, the 
nonzero part of any  point in the intersection $L\cap \stdCone_\gamma$ must be contained in the polyhedral 
portion $\widetilde P$. Therefore, Equation \eqref{eq_prop_lineality_separation} is satisfied and 
Proposition \ref{prop_strong_separation} ensures that   $\mathcal{F}$ and  $\stdCone_\gamma$ can 
be strongly separated. Thus the Euclidean distance between these two sets must be positive, which 
is a contradiction. 

\end{proof}

We can now extend the recipe discussed after Lemma \ref{lemma:preserv}.
Let $\{a^1, \ldots, a^\gamma\}$ the corresponding reducing directions which 
produces the cone $\stdCone_\gamma$. Pick an optimal solution $y^*$ for \eqref{eq:dual_red} and 
let $\hat y$ be any solution such that $c- \stdMap^T\hat y \in \reInt \stdCone_\gamma$. By induction, 
for a fixed $\beta \in [0,1)$ there are positive constants $\alpha_1, \ldots, \alpha _\gamma$ such that $z_\beta = c - \stdMap^T ((1-\beta) \hat y + \beta y^*) + \sum _{i=1}^\gamma 
\alpha _1 a^i$ corresponds to a feasible solution to \eqref{eq:dual} having value 
equal to $b^T{((1-\beta) \hat y + \beta y^*)}$. As before, when $\beta$ approaches $1$, the value 
of $z_\beta$ approaches $\dOpt$. This shows very clearly how the 
reducing directions can be used to construct a path towards optimality.
Moreover, the discussion on Section \ref{sec:min_d} shows that at most $m$ directions 
are needed to build points close to optimality, where $m$ is the number of Lorentz cones.

To wrap up this section, we remark that we proved similar results for SDP in Section 5 of \cite{lourenco_muramatsu_tsuchiya3}, where 
it is shown how to obtain a pair of primal and problems such that \emph{both} are strongly feasible. Therefore, 
a key difference here is that Theorem \ref{theo:attainment} does not ensure that (P') has a relative interior point. 
However, this is not necessary to ensure attainment, 
due to Proposition \ref{prop_strong_separation} and the fact that we have a mixture of nonlinear and polyhedral cones.

\section{Determining the Feasibility Status}\label{sec:back}
Let $(\stdCone,L,c)$ be an arbitrary SOCFP. According to 
Theorem \ref{theorem_reduction} the feasibility status of $(\stdCone,L,c)$ and 
the last problem $(\stdCone_\gamma, L,c)$ is exactly the same, except, perhaps, if 
$(\stdCone,L,c)$ is weakly infeasible. As mentioned in Section \ref{sec:notation}, there 
are simple (finite) certificates for three of the feasibility statuses and these 
are exactly the three statuses that are possible for $(\stdCone_\gamma, L,c)$. 

Therefore, to determine the feasibility status of $(\stdCone,L,c)$, one can first 
compute a relaxation sequence of $(\stdCone,L,c)$ and seek for appropriate certificates 
for $(\stdCone_\gamma, L,c)$. If $(\stdCone_\gamma, L,c)$ is weakly feasible then we have ahead of 
ourselves the task of distinguishing between weak feasibility and weak infeasibility of $(\stdCone,L,c)$.
To do that, it is enough to produce a finite certificate of infeasibility for $(\stdCone,L,c)$, 
which can be done through facial reduction \cite{Borwein1981495, borwein_facial_1981,pataki_strong_2013,article_waki_muramatsu} as follows.

The variant in \cite{article_waki_muramatsu} is capable of detecting infeasibility. 
 Starting with $F_0 = \stdCone$, facial reduction algorithms (FRAs) 
successively identify elements $d^i \in (F_{i-1}^* \setminus F_{i-i}^\perp )\cap L^\perp$ satisfying 
$\inProd{d^i}{c} \leq 0$. After $d^i$ is found, we define $F_i = F_{i-1} \cap \{d^i\}^\perp$ and 
repeat. At any step, if no $d^i$ exists, then the minimal face of $\stdCone$ which contains the 
feasible region is precisely $F_{i-1}$. As the $F_i$ form a strictly descending chain of faces of $\stdCone$, 
dimensional considerations readily imply that FRA must end in a finite number of steps.
Moreover, it can be shown  that $(K,L,c)$ is infeasible if and only if 
$\inProd{d^i}{c} < 0$ at some iteration.  For more details, see Section 3 in \cite{article_waki_muramatsu}.
If $\stdCone$ is an extended second order cone, it is possible to show that the search for 
$d^i$ can be cast as a SOCP as well. The upshot is that we may do facial reduction without ever 
leaving the SOCP world and the vectors $d^i$ serve as witnesses of the infeasibility of $(K,L,c)$.  

The certificate coming from facial reduction, the set of reducing directions associated 
to a relaxation sequence and a feasible solution to the last problem neatly summarize all aspects of 
weak infeasibility. Note that Theorem \ref{theorem_directions} and its proof  show that the  
directions and a feasible solution to the last problem can be used to construct 
points in $L+c$ which are arbitrarily close to $\stdCone$. This fulfills the goal of producing a 
finite certificate that $\text{dist}(\stdCone,L+c) = 0$ in spite of the fact that $\stdCone \cap (L+c) = \emptyset$. 

\newcommand{\replacedCone}{\widetilde K_p}

\begin{example}
Let $(\stdCone, L,c)$ be as in Example \ref{example_relaxation}. The last problem 
obtained was $(H^{3}_{a_{n_1}'} \times H^{3}_{b_{n_2}'},L,c)$, which was a weakly 
feasible problem. We now have to find out whether $(\stdCone, L,c)$ is weakly infeasible 
or feasible. In this case, it is easy because the point $(s,s,1)$ can never belong to 
$\SOC{3}$. If the problem were more complicated, one could formally do facial reduction and check 
its infeasibility. Note that, following Theorem \ref{theorem_directions} the vectors $a,b$ and $(0,0,1)\times (0,0,1)$ attest 
that $\text{dist}(\SOC{3}\times \SOC{3}, L+c) = 0$.

\end{example}

\section{Conclusion}\label{sec:conc}

In this paper, we presented an analysis of weakly infeasible problems via relaxation sequences. 
We used that to prove a basic result on the existence of affine subspaces that preserve 
the weak infeasibility of the problem (Theorem \ref{theorem_directions}) and we showed 
how to regularize a strongly feasible problem in order to guarantee that it is attained.

\bibliographystyle{spmpsci} 
\bibliography{bib}

\begin{thebibliography}{10}
\providecommand{\url}[1]{{#1}}
\providecommand{\urlprefix}{URL }
\expandafter\ifx\csname urlstyle\endcsname\relax
  \providecommand{\doi}[1]{DOI~\discretionary{}{}{}#1}\else
  \providecommand{\doi}{DOI~\discretionary{}{}{}\begingroup
  \urlstyle{rm}\Url}\fi

\bibitem{Borwein1981495}
Borwein, J., Wolkowicz, H.: Regularizing the abstract convex program.
\newblock Journal of Mathematical Analysis and Applications \textbf{83}(2), 495
  -- 530 (1981)

\bibitem{borwein_closedness_2009}
Borwein, J.M., Moors, W.B.: Stability of closedness of convex cones under
  linear mappings.
\newblock Journal of Convex Analysis \textbf{16}(3), 699--705 (2009)

\bibitem{borwein_facial_1981}
Borwein, J.M., Wolkowicz, H.: Facial reduction for a cone-convex programming
  problem.
\newblock Journal of the Australian Mathematical Society (Series A)
  \textbf{30}(03), 369--380 (1981)

\bibitem{klep_exact_2013}
Klep, I., Schweighofer, M.: An exact duality theory for semidefinite
  programming based on sums of squares.
\newblock Math. Oper. Res. \textbf{38}(3), 569–590 (2013)

\bibitem{liu_pataki_2014}
Liu, M., Pataki, G.: Exact duality in semidefinite programming based on
  elementary reformulations.
\newblock SIAM Journal on Optimization \textbf{25}(3), 1441--1454 (2015)

\bibitem{liu_pataki_2015}
Liu, M., Pataki, G.: Exact duals and short certificates of infeasibility and
  weak infeasibility in conic linear programming.
\newblock {O}ptimization {O}nline (2015)

\bibitem{Lobo1998193}
Lobo, M.S., Vandenberghe, L., Boyd, S., Lebret, H.: Applications of
  second-order cone programming.
\newblock Linear Algebra and its Applications \textbf{284}(1–3), 193 -- 228
  (1998)

\bibitem{lourenco_muramatsu_tsuchiya}
Louren\c{c}o, B.F., Muramatsu, M., Tsuchiya, T.: A structural geometrical
  analysis of weakly infeasible {SDP}s.
\newblock Tech. rep. (2013)

\bibitem{lourenco_muramatsu_tsuchiya3}
Louren\c{c}o, B.F., Muramatsu, M., Tsuchiya, T.: Solving {SDP} completely with
  an interior point oracle.
\newblock Tech. rep. (2015)

\bibitem{Luo96dualityand}
{Q}uan Luo, Z., Sturm, J.F., Zhang, S.: Duality and self-duality for conic
  convex programming.
\newblock Tech. rep., Econometric Institute, Erasmus University Rotterdam
  (1996)

\bibitem{Luo97dualityresults}
{Q}uan Luo, Z., Sturm, J.F., Zhang, S.: Duality results for conic convex
  programming.
\newblock Tech. rep., Econometric Institute, Erasmus University Rotterdam
  (1997)

\bibitem{Monteiro_Tsuchiya}
Monteiro, R.D., Tsuchiya, T.: Polynomial convergence of primal-dual algorithms
  for the second-order cone program based on the {MZ}-family of directions.
\newblock Mathematical Programming \textbf{88}(1), 61--83 (2000)

\bibitem{NesterovNemirovskii}
Nesterov, Y., Nemirovskii, A.: Interior-Point Polynomial Algorithms in Convex
  Programming.
\newblock Society for Industrial and Applied Mathematics (1994)

\bibitem{pataki_closedness_2007}
Pataki, G.: On the closedness of the linear image of a closed convex cone.
\newblock Mathematics of Operations Research \textbf{32}(2), 395--412 (2007)

\bibitem{pataki_strong_2013}
Pataki, G.: Strong duality in conic linear programming: Facial reduction and
  extended duals.
\newblock In: Computational and Analytical Mathematics, vol.~50, pp. 613--634.
  Springer New York (2013)

\bibitem{polik_new_2009}
P\'olik, I., Terlaky, T.: New stopping criteria for detecting infeasibility in
  conic optimization.
\newblock Optimization Letters \textbf{3}(2), 187--198 (2009)

\bibitem{rockafellar}
Rockafellar, R.T.: {Convex Analysis }.
\newblock Princeton University Press (1997)

\bibitem{Tsuchiya99}
Tsuchiya, T.: A convergence analysis of the scaling-invariant primal-dual
  path-following algorithms for second-order cone programming.
\newblock Optimization Methods and Software \textbf{11}(1-4), 141--182 (1999)

\bibitem{waki_how_2012}
Waki, H.: How to generate weakly infeasible semidefinite programs via
  {L}asserre’s relaxations for polynomial optimization.
\newblock Optimization Letters \textbf{6}(8), 1883--1896 (2012)

\bibitem{article_waki_muramatsu}
Waki, H., Muramatsu, M.: Facial reduction algorithms for conic optimization
  problems.
\newblock Journal of Optimization Theory and Applications \textbf{158}(1),
  188--215 (2013)

\end{thebibliography}
\end{document}